\documentclass{amsart}

\usepackage{amsmath}

\title{The Neretin groups}

\author{{\L}ukasz Garncarek}
\author{Nir Lazarovich}

\thanks{During the work on this paper the first author was supported
  by a scholarship of the Foundation for Polish Science and by the
  grant 2012/06/A/ST1/00259 of the National Science Center. The work
  was conducted during the first author's internship at the Warsaw
  Center of Mathematics and Computer Science.}

\newtheorem{theorem}{Theorem}[section]
\newtheorem{lemma}[theorem]{Lemma}
\newtheorem{proposition}[theorem]{Proposition}

\theoremstyle{definition}
\newtheorem{definition}{Definition}

\theoremstyle{remark}
\newtheorem{remark}[theorem]{Remark}

\newcommand{\NN}{\mathbb{N}}
\newcommand{\ZZ}{\mathbb{Z}}
\newcommand{\abs}[1]{\left\lvert{#1}\right\rvert}
\newcommand{\bd}{\partial}

\DeclareMathOperator{\Homeo}{Homeo}
\DeclareMathOperator{\Aut}{Aut}
\DeclareMathOperator{\diam}{diam}
\DeclareMathOperator{\sgn}{sgn}
\DeclareMathOperator{\supp}{supp}

\begin{document}
\maketitle

\section{Introduction}
\label{sec:introduction}

The Neretin group $N_q$ was introduced in \cite{Neretin1993} as an
analogue of the diffeomorphism group of the circle. It is a subgroup
of the homeomorphism group of the boundary of an infinite $q$-regular
tree $T$, consisting of elements which locally act by similarities of
the visual metric. We define the Neretin group, endow it with a group
topology, and present the proof of its simplicity, following
\cite{Kapoudjian}.

In Section \ref{sec:preliminaries} we discuss the structure of the
boundary of a regular tree. Sections \ref{sec:neretin-groups} and
\ref{sec:topol-neret-group} define the Neretin group, and describe its
locally compact totally disconnected group topology. In Sections
\ref{sec:higm-thomps-groups} and \ref{sec:conv-gener-set} we define
the Higman-Thompson groups $G_{q,r}$, another family of groups related to
boundaries of regular trees, and show how they can be embedded into the
Neretin group. Then we prove that the Neretin group $N_q$ is generated by
any of the embedded copies of the Higman-Thompson group $G_{q,2}$
together with the canonically embedded group of type-preserving
automorphisms of the tree $T$. Finally, Section~\ref{sec:simplicity}
presents the proof of simplicity of the Neretin groups.

We do not make any claims of originality of the presented
results. This text is an extended summary of our talk given during the
Oberwolfach Arbeitsgemeinschaft on totally disconnected locally
compact groups, held in October 2014. 

\section{Preliminaries}
\label{sec:preliminaries}

A tree $T$ is a nonempty connected undirected simple graph without
nontrivial cycles.  We will interchangeably treat $T$ as a set of
vertices endowed with a binary relation of adjacency, or as a
topological space obtained from the set of vertices by gluing in unit
intervals corresponding to edges. A fixed basepoint $o\in T$ defines a
partial order on $T$, namely $v\leq_o w$ if and only if the path from
$v$ to $o$ passes through $w$. The basepoint $o$ is the greatest
element in this order. The tree structure on $T$ can be recovered from
the poset $(T,\leq_o)$ as follows. Two elements $v,w\in T$ are adjacent
if and only if they are comparable, and there are no other elements
between them. It follows that tree automorphisms of $T$ fixing $o$ are
exactly the order automorphisms of $(T,\leq_o)$.

The distance between vertices $v,w\in T$, i.e.\ the number of edges on
the unique path joining them, will be denoted by $\abs{vw}$. A vertex
$v\in T$ of degree $1$ is called a \emph{leaf}.  A tree is
\emph{$q$-regular} for some $q\in\NN$ if all its vertices have degree
$q+1$. A \emph{finite $q$-regular tree} is a finite tree, whose every
vertex is either a leaf, or has degree $q+1$, in which case we call it
\emph{internal}. A \emph{rooted tree} is a tree with a fixed choice of
base vertex $o\in T$, called its \emph{root}. In case of rooted trees
we slightly modify the definition of $q$-regularity by requiring the
root to be of degree $q$ instead of $q+1$. 
In the subsequent sections we will deal only with regular trees, so
let us assume from now on that $T$ is a rooted or unrooted $q$-regular
tree with $q\geq 2$. This will relieve us from considering some
special cases, which would otherwise appear in the following
discussion.

By a \emph{ray} in $T$ we understand an infinite
path, i.e.\ a sequence $(v_0,v_1,\ldots)$ of distinct vertices of $T$
such that the consecutive ones are connected by edges. Two rays are
said to be \emph{asymptotic} if, after removing some finite initial
subsequences, they become equal. Equivalence classes of rays in $T$
are called the \emph{ends} of $T$. The set of all ends of $T$ is
denoted by $\bd T$ and referred to as the \emph{boundary} of $T$. Any
end $\xi\in\bd T$ has a unique representative $\xi_v$ with a given
initial vertex $v\in T$. To see it, one has to pick a representative
$(v_0,v_1,\ldots)$ of $\xi$, find a minimal path from $v$ to one of
the vertices $v_i$, and replace the initial segment $(v_0,\ldots,v_i)$
by this path. Thus, if we choose a base vertex $o\in T$, we may
identify $\bd T$ with the set of rays emanating from $o$.

For $o\in T$ denote by $(\xi,\eta)_o$ the length of the common initial
segment of the representatives $\xi_o$ and $\eta_o$ of two ends
$\xi,\eta\in\bd T$. It takes values in $\NN\cup\{\infty\}$. Together
with a choice of $\epsilon > 0$ this allows to define a \emph{visual metric}
$d_{o,\epsilon}$ on $\bd T$ by
\begin{equation}
  \label{eq:def-metric}
  d_{o,\epsilon}(\xi,\eta) = e^{-\epsilon(\xi,\eta)_o}.
\end{equation}
The change of the basepoint $o$ leads to a bi-Lipschitz equivalent
metric, and changing $\epsilon$ still gives the same topology. It is
an exercise to check that this unique natural topology on $\bd T$ is
compact and second countable, provided that $T$ is \emph{locally
  finite}, i.e.\ every vertex has finite degree. 

If $j\colon T_1 \to T_2$ is an embedding of trees, it sends rays to
rays, and preserves asymptoticity. Hence, it induces a map $j_*\colon
\bd T_1 \to \bd T_2$. If we choose basepoints $o_i\in T_i$ in such a
way that $j(o_1)=o_2$, then $j_*$ can be seen to be an isometric
embedding, so in particular it is continuous. Taking the boundary in
fact gives a functor from the category of trees and tree embeddings
into the category of topological spaces and continuous embeddings.

From now on we will fix $\epsilon=1$ and $o\in T$, and suppress them
from notation whenever possible. There may exist more natural choices
for $\epsilon$, e.g.\ in the case of regular trees, but they will not
be of any use to us. The geometrically obvious inequality $(\xi,\eta)
\geq \min\{ (\xi,\zeta),(\zeta,\eta)\}$ implies that $d$ is in fact an
ultrametric, i.e.\ it satisfies a stronger variant of the triangle
condition,
\begin{equation}
  \label{eq:visual-ultrametric}
  d(\xi,\eta) \leq \max \{ d(\xi,\zeta), d(\zeta,\eta) \}.
\end{equation}
As a consequence, two open balls in $(\bd T,d)$ are either disjoint, or one
of them is contained in the other. It follows that the covering of
$\bd T$ by open balls of fixed radius is in fact a partition into
open---and hence also closed---sets, and $\bd T$ is totally
disconnected. Additionally, since the metric $d$ takes values in a
discrete set, any closed ball is also an open ball with a slightly
larger radius, and vice versa.

Ultrametricity implies that any point of a ball in $(\bd T,d)$ is its
center. It is however still possible to effectively enumerate the
balls in a one-to-one manner. To this end, for $v\in T$ define $T_v$
as the subtree of $T$ spanned by the vertices $\{w\in T : w \leq_o
v\}$. It is a rooted $q$-regular tree with root $v$, and its boundary
$\bd T_v$ is a subset of $\bd T$. It is in fact a closed ball of
radius equals $e^{-\abs{ov}}$, and the embedding $(\bd T_v, d_v) \to
(\bd T, d_o)$ is a similarity. On the other hand, any ball
$B\subseteq\bd T$ is a closed ball of radius $e^{-n}$ for some
$n\in\NN$, and can be written as $\bd T_v$ where $v$ is the last
vertex of the common initial segment of all the rays $\xi_o$
representing points $\xi\in B$. The family of non-empty balls in $\bd
T$ is therefore in a one-to-one correspondence with vertices of $T$.

As a consequence of the above discussion, the assignment $v\mapsto \bd
T_v$ is an order-isomorphism between $(T,\leq_o)$ and the set
$\mathcal{B}(\bd T, d_0)$ of all balls in $(\bd T, d_o)$ ordered by
inclusion. Moreover, if $\phi\colon T\to T'$ is a
basepoint-preserving isomorphism of trees, then $\phi(T_v)=T_{\phi(v)}$, and
\begin{equation}
  \phi_*(\bd T_v) = \bd \phi(T_v) = \bd T_{\phi(v)}.
\end{equation}
This means that the order-isomorphism between $\mathcal{B}(\bd T,d_o)$
and $\mathcal{B}(\bd T',d_{\phi(o)})$ induced by $\phi$ is the same as
the one induced by $\phi_* \colon \bd T\to\bd T'$. This correspondence
can be reversed, namely if $\Phi\colon \bd T\to \bd T'$ is a
homeomorphism preserving balls, it necessarily preserves their
inclusion, and induces an order-isomorphism of $\mathcal{B}(\bd
T,d_o)$ and $\mathcal{B}(\bd T',d_{\phi(o)})$ yielding a
basepoint-preserving isomorphism $\phi\colon T\to T'$. It satisfies
\begin{equation}
  \phi_*(\bd T_v) = \bd T_{\phi(v)} = \Phi(\bd T_v), 
\end{equation}
but since the balls form a basis of the topology of $\bd T$, this
means that $\Phi=\phi_*$.


Finally, let us introduce the notion of a \emph{forest}. It is what we
obtain if we remove the assumption of connectedness from the
definition of a tree. In other words, a forest is a graph $F$ which
decomposes into a disjoint union of trees. We may define its boundary
$\bd F$ as the disjoint union of the boundaries of its constituent
trees. It is again functorial. Most of the discussion above extends to
forests.

\section{The Neretin groups of spheromorphisms}
\label{sec:neretin-groups}

Let $T$ be a $q$-regular tree. For a nonempty finite $q$-regular
subtree $F\subseteq T$, by the difference $T\setminus F$ we will
understand the rooted $q$-regular forest obtained by removing from $T$
all the edges and internal vertices of $F$, and designating the leaves
of $F$ as the roots; geometrically, this amounts to removing the
interior of $F$ from $T$. Clearly, $\bd(T\setminus F)=\bd T$, as
every ray in $T$ has a subray disjoint from $F$.

Now, let $F_1,F_2\subseteq T$ be two finite $q$-regular subtrees, such
that there exists an isomorphism of forests $\phi\colon T\setminus
F_1\to T\setminus F_2$. It induces a homeomorphism $\phi_*$ of $\bd
T$, called a \emph{spheromorphism} of $\bd T$. The isomorphism $\phi$
will be referred to as a representative of $\phi_*$.

Observe that the identity map of $\bd T$ is a spheromorphism. More
generally, if $\phi\in\Aut(T)$, then for any subtree $F\subseteq T$
the map $\phi$ restricts to an isomorphism of forests $T\setminus F
\to T\setminus \phi(F)$, and thus the induced homeomorphism $\phi_*$
is a spheromorphism. Moreover, the inverse of a spheromorphism is also
a spheromorphism, and for any pair of spheromorphisms $\phi_*$ and
$\psi_*$ we may find representatives which are composable, showing
that $\psi_*\circ \phi_*$ is also a spheromorphism.

\begin{definition} 
  \label{def:neretin-group}
  The \emph{Neretin group} $N_q$ is the group of all spheromorphisms
  of the boundary of a $q$-regular tree.
\end{definition}

The group $N_q$ has another description, based upon the metric
structure of the boundary. We will call a homeomorphism of metric
spaces $\Phi\colon X\to Y$ a \emph{local similarity} if for each $x\in
X$ there exists an open neighborhood $U$ of $x$ and a constant $\lambda_U>0$
such that for every $x_1,x_2\in U$ we have
\begin{equation}
  d_Y(\Phi(x_1),\Phi(x_2)) = \lambda_U d_X(x_1,x_2),
\end{equation}
i.e.\ the restriction $\Phi|_U\colon U\to \Phi(U)$ is a
similarity \cite{Hughes2009}. The requirement that $\phi$ is a homeomorphism allows to
choose $U$ to be a ball $B(x,r)$ centered at $x$, such that
$\Phi(B(x,r))=B(\Phi(x),\lambda_Ur)$. It is clear that all local
similarities of a metric space form a group.

\begin{proposition}
  \label{prop:neretin-local-sim}
  For a homeomorphism $\Phi\in\Homeo(\bd T)$ the following conditions
  are equivalent.
  \begin{enumerate}
  \item $\Phi$ is a spheromorphism,
  \item $\Phi$ is a local similarity with respect to any visual metric
    on $\bd T$,
  \item $\Phi$ is a local similarity with respect to some visual
    metric on $\bd T$.
  \end{enumerate}
\end{proposition}

\begin{proof}
  We begin by showing that (1) implies (2). Fix a basepoint $o\in T$
  and the corresponding visual metric $d$. Let $\Phi=\phi_*$ be a
  spheromorphism represented by $\phi\colon T\setminus F_1\to
  T\setminus F_2$. We may assume that both $F_1$ and $F_2$ contain $o$
  as internal vertex. Let $T\setminus F_1 = T_1 \cup \cdots \cup T_k$
  be the decomposition into disjoint trees. Then $T\setminus F_2$
  decomposes into $\phi(T_1) \cup \cdots \cup \phi(T_k)$. These
  decompositions induce partitions of $\bd T$ into open balls $\bd
  T_i$ and $\bd (\phi(T_i)) = \phi_*(\bd T_i)$. 

  Let $v$ be the root of $T_i$. Then the root of
  $\phi(T_i)$ is $\phi(v)$. For $\xi,\eta\in\bd T_i\subseteq\bd T$ we
  have
  \begin{equation}
    \begin{split}
      (\phi_*(\xi),\phi_*(\eta))_o & =
      (\phi_*(\xi),\phi_*(\eta))_{\phi(v)} + \abs{o\phi(v)} = \\
      & = (\xi,\eta)_v + \abs{o\phi(v)} = (\xi,\eta)_o - \abs{ov} +
      \abs{o\phi(v)},
    \end{split}
  \end{equation}
  which implies that 
  \begin{equation}
    d(\phi_*(\xi),\phi_*(\eta)) = e^{\abs{ov}-\abs{o\phi(v)}} d(\xi,\eta),
  \end{equation}
  and $\phi_*|_{\bd T_i} \colon \bd T_i \to \phi_*(\bd T_i)$ is a
  similarity.

  The other nontrivial implication is from (3) to (1). Let
  $\Phi\in\Homeo(\bd T)$ be a local similarity of $(\bd T, d_o)$.  By
  compactness, we may cover $\bd T$ by finitely many balls $B$ on
  which $\Phi$ is a similarity, and $\Phi(B)$ is also a ball. By
  ultrametricity, we may assume that this covering is disjoint, and
  contains at least $2$ balls. The balls in the covering are of the
  form $\bd T_v$ with $v$ in some finite set $L\subseteq T$, and
  $\Phi(\bd T_v)=\bd T_{v'}$ for some $v'\in T$. The restriction
  $\Phi|_{\bd T_v}\colon \bd T_v \to \bd T_{v'}$ preserves balls, and
  therefore is induced by a root-preserving isomorphism $\phi_v \colon
  T_v\to T_{v'}$. It now remains to observe that the forests $\bigcup
  T_v$ and $\bigcup T_{v'}$ are obtained by removing finite regular
  subtrees from $T$, so that the isomorphisms $\phi_v$ assemble into
  an isomorphism of these forests, representing a spheromorphism.
\end{proof}

\section{Topology on the Neretin groups}
\label{sec:topol-neret-group}

The Neretin group $N_q$ is a subgroup of the homeomorphism group of
$\bd T$, which is a topological group when endowed with the
compact-open topology. Since $\bd T$ is compact, this topology is
metrizable: for $\Phi,\Psi\in\Homeo(\bd T)$ we have
\begin{equation}
  d(\Phi, \Psi) = \sup_{\xi\in\bd T} d_o(\Phi(\xi),\Psi(\xi)).
\end{equation}
A first choice for the group topology on $N_q$ would be to restrict
the compact-open topology. Unfortunately, this restriction is not locally
compact, as we will now observe, using the following lemma.

\begin{lemma}
  \label{lem:lc-sbgp-closed}
  If a subgroup $H$ of a topological group $G$ is locally compact,
  then it is closed.
\end{lemma}

\begin{proof}
  First, assume that $H$ is dense in $G$. Let $U$ be an open
  neighborhood of $1$ in $G$ such that the closure $K$ of $U\cap H$ in
  $H$ is compact. We then have $K\cap U = H\cap U$. This set is both
  closed an dense in $U$, hence it is equal to $U$. Therefore
  $U\subseteq H$, so the subgroup $H$ is open, and hence closed.

  In the general case $H$ is dense, and hence closed in its
  closure in $G$. This means that it is closed in $G$.
\end{proof}

Now, we can see that if $N_q$ with the compact-open topology was
locally compact, it would be a closed subgroup. We will show this is
false by constructing a sequence of spheromorphisms converging to a
homeomorphism outside $N_q$, using the description of spheromorphisms
as local similarities. Let $B_i$ be a sequence of pairwise disjoint
balls in $\bd T$. For each $i$ we may construct a spheromorphism
$\Phi_i\in N_q$, which is identity outside $B_i$, and on some ball
inside $B_i$ it restricts to a similarity with scale greater that
$i$. The sequence $\Psi_k=\Phi_k \circ\cdots\circ \Phi_1$ is Cauchy,
as for $k>l$
\begin{equation}
  d(\Psi_k,\Psi_l) = d(\Phi_k\circ\cdots\circ \Phi_{l+1}, \mathrm{id})
  \leq \max_{l < i \leq k} \diam B_i \xrightarrow[l\to \infty]{} 0.
\end{equation}
Therefore, $\Psi_k$ converge to a homeomorphism $\Psi\in\Homeo(\bd
T)$. It has the same restriction to $B_i$ as $\Phi_i$, and therefore
on some ball it restricts to a similarity of scale at least $i$. But
local similarities are Lipschitz, so $\Psi\not\in N_q$.

The issue of endowing $N_q$ with a locally compact group topology can
be resolved by observing that it already contains a locally compact
group as a subgroup. Indeed, $\Aut(T)$ naturally embeds in $N_q$ (and
will be identified with its image) and carries the compact-open
topology coming from its action on the set of vertices of $T$. It is
locally compact and totally disconnected. We can extend it to $N_q$ by
declaring the left cosets of $\Aut(T)$ to be open and homeomorphic to
$\Aut(T)$ by the translation maps. If $g\!\Aut(T)=h\!\Aut(T)$, then
the translation maps induce the same topology, so it is well defined,
and clearly locally compact and totally disconnected. What is not so
clear is whether this makes $N_q$ a topological group---if it does,
then this topology is clearly the unique one making $\Aut(T)$ with its
original topology an open subgroup of $N_q$. This issue is addressed
by the following two lemmas.

\begin{lemma}
  \label{lem:general-existence-of-topology}
  Suppose that an abstract group $G$ contains a topological group $H$
  as a subgroup. Then $G$ admits a unique group topology, in which $H$
  becomes an open subgroup, provided that for all open subsets
  $U\subseteq H$ and $g,g'\in G$ the intersection $gUg'\cap H$ is open
  in $H$.
\end{lemma}

\begin{proof}
  The basis for the topology on $G$ is necessarily the family of left
  translates of open subsets of $H$. This indeed makes $H$ embedded
  homeomorphically as an open subset. Moreover, right translates of
  open subsets of $H$ are also open, since for $U\subseteq H$ open and
  $g'\in G$ the set $Ug'$ can be written as
  \begin{equation}
    Ug' = \bigcup_{g\in G} (gH \cap Ug') = \bigcup_{g\in G} g(H\cap g^{-1}Ug'),
  \end{equation}
  which is a union of left translates of open subsets of $H$. As a
  consequence, left and right translations are homeomorphisms of $G$.

  Now we need to ensure that multiplication and inversion are
  continuous. Let $g_\alpha\to g$ and $g'_\alpha\to g'$ be two
  convergent nets in $G$. The cosets $gH$ and $Hg'$ are open
  neighborhoods of $g$ and $g'$ respectively, so without loss of
  generality we may assume that $g_\alpha =gh_\alpha$ and $g'_\alpha =
  h'_\alpha g'$ with $h_\alpha,h'_\alpha \in H$ converging to
  $1$. Then $h_\alpha h'_\alpha\to 1$ in $H$, and therefore $g_\alpha
  g'_\alpha = gh_\alpha h'_\alpha g' \to gg'$. Similarly
  $g_\alpha^{-1} = h_\alpha^{-1} g^{-1} \to g^{-1}$.
\end{proof}

In order to show that the topology we put on $N_q$ is indeed a group
topology, it remains to show that for every open $U\subseteq \Aut(T)$
and $g,g'\in N_q$ the subset $\Aut(T) \cap gUg'$ is open in
$\Aut(T)$. Observe that this property is preserved under unions and
finite intersections, so it is enough to show it for $U$ in a certain
subbasis of the topology on $\Aut(T)$. 

Let $o\in T$ be a base vertex, and denote by $K$ the stabilizer of $o$
in $\Aut(T)$. Then for $g,h \in \Aut(T)$ the set $gKh$ consists
exactly of the automorphisms sending $h^{-1}(o)$ to $g(o)$, so the
finite intersections of the sets of the form $gKh$ yield the standard
basis for the topology of $\Aut(T)$. We are thus left with proving the
following.

\begin{lemma}
  \label{lem:neretin-topological-condition}
  If $K$ is the stabilizer of the base vertex $o\in T$ in the group
  $\Aut(T)$, then for all $\phi_*,\psi_* \in N_q$ the intersection $\psi_*K\phi_* \cap
  \Aut(T)$ is open in $Aut(T)$.
\end{lemma}

\begin{proof}
  The spheromorphisms $\phi_*$ and $\psi_*$ admit representatives
  $\phi\colon T\setminus F_1\to T\setminus B$ and $\psi\colon
  T\setminus B\to T\setminus F_2$, where $B$ is a sufficiently large
  ball in $T$, centered at $o$. Let $K_B\subseteq K$ denote the
  pointwise stabilizer of this ball; it is an open subgroup of
  $\Aut(T)$. We have
  \begin{equation}
    \psi_*K\phi_* = \bigcup_{k \in K} \psi_*kK_B\phi_* =
    \bigcup_{k \in K} \psi_*k\phi_*(\phi^{-1}_*K_B\phi_*),
  \end{equation}
  where $\phi^{-1}_*K_B\phi_*$ consists of elements $\eta_*$ whose
  representatives $\eta\colon T\setminus F_1\to T\setminus F_1$ leave
  the trees of the forest $T\setminus F_1$ in place, and thus extend
  to automorphisms of $T$. Hence, it is an open subgroup of $\Aut(T)$,
  namely the pointwise stabilizer of $F_1$, and therefore the
  intersection $\psi_*K\phi_* \cap \Aut(T)$ is open.
\end{proof}

This shows that the topology we defined on $N_q$ is indeed a group
topology. We may summarize this as follows.
\begin{theorem}
  \label{thm:topology-for-neretin-gp}
  The Neretin group $N_q$ admits a unique group topology such that the
  natural embedding $\Aut(T) \to N_q$ is continuous and open. With
  this topology, $N_q$ is a totally disconnected locally compact
  group.
\end{theorem}

\section{The Higman-Thompson groups}
\label{sec:higm-thomps-groups}

A tree $T$ is \emph{planar} if it is rooted and for every $v\in T$
there is a fixed linear order on the set of children of $v$. This
corresponds to specifying a way to draw the tree on the plane, so that
for every $v\in T$ its children are below it, ordered from left to
right. The structure of a planar tree is very rigid---an
isomorphism of planar trees, which is required to preserve the
roots and orders on the sets of children, is always unique, if it
exists.

Let $\mathcal{F}$ be a forest consisting of $r$ planar $q$-regular
trees $T_1,\ldots, T_r$. For every $i$ choose rooted (in particular,
this implies that $F_i$ and $F'_i$ have the same root as $T_i$) finite
regular subtrees $F_i,F'_i\subseteq T_i$ in such a way that the
forests $\mathcal{F}_1=\mathcal{F}\setminus \bigcup F_i$ and
$\mathcal{F}_2=\mathcal{F}\setminus\bigcup F'_i$ have the same
number of trees. The forests $\mathcal{F}_1$ and $\mathcal{F}_2$
consist of planar $q$-regular trees, and hence for every bijection of
the sets of trees in $\mathcal{F}_1$ and $\mathcal{F}_2$ there exists
a unique isomorphism $\phi\colon \mathcal{F}_1\to \mathcal{F}_2$ of
planar forests realizing it. It induces a homeomorphism $\phi_*$ of
$\bd \mathcal{F}$, and the subgroup of $\Homeo(\bd \mathcal{F})$
containing all such homeomorphisms is called the Higman-Thompson
group $G_{q,r}$. The group $G_{2,1}$ is known as Thompson group
$V$. 

This definition shows some ties between $G_{q,r}$ and the permutation
groups $S_n$, which we will now make more explicit. The order of children on
each $T_i$ induces a lexicographic order on paths starting from the
root, which correspond to vertices. This defines a linear order on the
set of vertices of each $T_i$. Moreover, the trees themselves can be
ordered from $T_1$ to $T_r$, so we have a linear order on the set of
vertices of $\mathcal{F}$. This allows to order the trees in
$\mathcal{F}_1$ and $\mathcal{F}_2$ by looking at the order of their
roots. An isomorphism $\phi\colon \mathcal{F}_1\to \mathcal{F}_2$ of
planar forests is now completely determined by a permutation
$\sigma\in S_n$, where $n$ is the number of trees in the forests
$\mathcal{F}_i$.

We will use this to define a homomorphism $\theta \colon G_{q,r}\to
\ZZ/2\ZZ$. If $q$ is even, $\theta$ is just the zero homomorphism. On
the other hand, if $q$ is odd, we claim that the sign of the
permutation $\sigma$ associated to $\phi$ in the above discussion
depends only on the element $\phi_*\in G_{n,r}$, and we put
$\theta(\phi_*)=\sgn \sigma$. To see this, observe that if
$\mathcal{F}_1=\bigcup_{i=1}^n L_i$ and $\mathcal{F}_2=\bigcup_{i=1}^n
L'_i$ are decompositions into trees, numbered in accordance with the
order, we may modify the representative $\phi$ in an elementary way as
follows. Choose one of the trees $L_i$ and remove its root, replacing
it with $q$ new trees. Do the same with
$\phi(L_i)=L'_{\sigma(i)}$. This gives a new representative $\phi'$ of
$\phi_*$, obtained by restricting $\phi$. The number of inversions
$I(\sigma')$ in the permutation $\sigma'$ associated to $\phi'$ is
equal to
\begin{equation}
  \begin{split}
    I(\sigma')  &=  \abs{\{ (j,k)\in\{1,\ldots,n\}^2 : j\ne i\ne k,\, j <
    k,\, \sigma(j)>\sigma(k)\}}+\\
    &\quad+ q\abs{\{ k\in\{i+1,\ldots,n\} : \sigma(k)<i \}} +\\
    &\quad+ q\abs{\{ j\in\{1,\ldots,i-1\} : \sigma(j)>i \}} =\\
    &=  I(\sigma) + (q-1)C,
  \end{split}
\end{equation}
where $(q-1)C$ is even. This means that the sign of the permutation
does not change when we apply the described elementary modification to
a representative of $\phi_*$. It remains to observe that any two
representatives of $\phi_*$ can be transformed by a sequence of elementary
modifications into the same third representative.

Using the homomorphism $\theta$ defined above, we may now describe the
commutator subgroup of $G_{q,r}$. The argument below is based on an
idea of Mati Rubin \cite{Brin2004}.
\begin{theorem}
  \label{thm:commutator-higman-thompson}
  The commutator subgroup $G'_{q,r}$ of the Higman-Thompson group $G_{q,r}$ is
  equal to the kernel of the homomorphism $\theta\colon G_{q,r}\to
  \ZZ/2\ZZ$. It is a simple group.
\end{theorem}

\begin{proof}


  It is clear that $G'_{q,r}\subseteq \ker \theta$, and we need to
  prove the opposite inclusion. First, we claim that $G_{q,r}$ is
  generated by elements with representatives $\phi\colon
  \mathcal{F}_1\to \mathcal{F}_2$ such that
  $\mathcal{F}_1=\mathcal{F}_2$. Indeed, if $\phi_*\in G_{q,r}$ is
  represented by $\phi\colon\mathcal{F}_1\to \mathcal{F}_2$, in both
  $\mathcal{F}_1$ and $\mathcal{F}_2$ we may find families of $q$
  trees whose roots have the same parent in $\mathcal{F}$. If we
  compose $\phi$ with a suitable $\psi\colon \mathcal{F}_2\to
  \mathcal{F}_2$ such that $\psi\circ\phi$ sends the $q$ fixed trees
  from $\mathcal{F}_1$ to the $q$ fixed trees in $\mathcal{F}_2$ in an
  order preserving way, then $(\psi\circ\phi)_*$ can be represented by
  a map $\chi\colon \mathcal{F}_1'\to \mathcal{F}_2'$, where
  $\mathcal{F}_i'$ is obtained from $\mathcal{F}_i$ by adding the
  common parent of the fixed $q$ trees, and joining them into a single
  tree. This process stops after finitely many steps, yielding a
  decomposition of $\phi_*$ into a product of the claimed generators.

  An element of $G_{q,r}$ supported in a proper subset of $\bd
  \mathcal{F}$, represented by $\phi\colon \mathcal{F}_1\to
  \mathcal{F}_1$ which exchanges two trees of $\mathcal{F}_1$ and
  leaves the rest in place, will be called a \emph{transposition}. It
  is now clear that $G_{q,r}$ is generated by transpositions, and if
  $q$ is odd, then $\ker \theta$ is generated by products of pairs of
  transpositions, which can be further assumed to have disjoint
  supports not covering the whole boundary $\bd\mathcal{F}$ (we will
  always assume this when speaking about a product of a pair of transpositions). 
  Moreover, if $q$ is even, any transposition can be decomposed into a
  product of an even number of transpositions with disjoint
  supports. Hence, $\ker\theta$ is always generated by products of
  pairs of disjoint transpositions.

  Now, observe that any two products of pairs of transpositions are
  conjugate in $G_{q,r}$. Thus, in order to complete the proof of the
  inclusion we need to show that the commutator subgroup $G_{q,r}'$
  contains a product of two disjoint transpositions, supported in a
  proper subset of $\bd \mathcal{F}$. To this end, we just need to
  take the commutator of two transpositions, one exchanging two balls
  in $\bd \mathcal{F}$, and the other supported inside one of these
  balls.


  We are now left with observing that the commutator subgroup
  $G_{q,r}'$ is simple. Let $N$ be a normal subgroup of $G_{q,r}'$
  containing a nontrivial element $\phi_*$.  There exists an open set
  $U\subseteq \bd\mathcal{F}$ such that $\phi_*(U)$ is disjoint from
  $U$, and $U\cup\phi_*(U)$ is a proper subset of $\bd\mathcal{F}$. If
  $\psi_*$ is a product of two disjoint transpositions supported in
  $U$ then the commutator $[\phi_*,\psi_*]$ is a product of four
  disjoint transpositions, is supported in $U\cup\phi_*(U)$, and
  belongs to $N$. If $\chi_*$ is a transposition supported outside
  $U\cup \phi_*(U)$, then $[\phi_*,\psi_*]$ is invariant under
  conjugation by $\chi_*$. Since any product of four disjoint
  transpositions, supported in a proper subset of $\bd \mathcal{F}$,
  is conjugate to $[\phi_*,\psi_*]$ by an element of
  $G_{q,r}=G_{q,r}'\cup \chi_* G_{q,r}'$, the normal subgroup $N$
  contains all such products. But from simplicity of the alternating
  groups it follows that for $n\geq 8$ the alternating group $A_n$ is
  generated by products of four disjoint transpositions, hence $N$
  contains all elements possessing a representative whose associated
  permutation is even. This means that $N=G_{q,r}'$.
\end{proof}

\section{A convenient generating set for $N_q$}
\label{sec:conv-gener-set}

Fix a $2$-coloring of vertices of $T$; in this context one usually
refers to colors as \emph{types}. An automorphism $\phi$ of $T$ is
\emph{type-preserving} if and only if whenever $\phi$ fixes an edge of
$T$, then it fixes its endpoints. The subgroup of $\Aut(T)$ consisting
of type-preserving automorphisms is denoted by $\Aut^+(T)$.

It is generated by the union of pointwise edge stabilizers in
$\Aut(T)$. Indeed, denote by $H$ the subgroup of $\Aut(T)$ generated
by the edge stabilizers. It is clearly a subgroup of $\Aut^+(T)$.
Any two edges with a common endpoint can be exchanged using an element
of the stabilizer of another edge with the same endopint. It follows
that all edges with a common endpoint lie in the same orbit of $H$,
and thus $H$ acts transitively on the edges of $T$. Hence, if $\phi\in
\Aut^+(T)$ and $e$ is an edge of $T$, then there exists $\psi\in H$
such that $\psi(e)=\phi(e)$. The element $\psi^{-1}\phi\in\Aut^+(T)$
fixes the edge $e$ pointwise, so $\psi^{-1}\phi\in H$.

As an instance of the Tits Simplicity Theorem \cite{Tits}, we obtain
the following.
\begin{theorem}
  \label{Tits simplicity}
  The group $\Aut^+(T)$ of type-preserving automorphisms is simple.
\end{theorem}

Now, consider the Higman-Thompson group $G_{q,2}$ acting on the
boundary of the planar forest $\mathcal{F}$. Pick an edge $e$ of $T$
and an embedding $i\colon \mathcal{F}\to T$ sending $\mathcal{F}$ onto
$T\setminus e$. It defines an embedding of $G_{q,2}$ into $N_q$ given
by $\Phi \mapsto i_*\circ \Phi \circ i_*^{-1}$, whose image we will
denote by $G_{q,2}^i$. If $j\colon \mathcal{F}\to T$ is another such
embedding, with image $T\setminus e'$, it can be written as $j = \eta
\circ i \circ \epsilon$ where $\eta\in\Aut^+(T)$, and $\epsilon$ is
either the identity map, or the unique automorphism of $\mathcal{F}$
preserving its structure of a planar forest, and exchanging its two
trees. As a consequence, in $N_q$ the subgroup $G_{q,2}^{j}$ is
conjugate to $G_{q,2}^i$ by an element of $\Aut^+(T)$.

\begin{lemma}
  \label{lem:generating-set}
  The Neretin group $N_q$ is generated by $\Aut^+(T)$ and $G_{q,2}^i$,
  for any embedding $i$.
\end{lemma}

\begin{proof}
  Let $\phi_*\in N_q$ be represented by $\phi\colon T\setminus F_1\to
  T\setminus F_2$. We may suppose that the subtrees $F_1$ and $F_2$
  share a common edge $e$. Choose an isomorphism $j\colon \mathcal{F}
  \to T\setminus e$; it defines the subgroup $G_{q,2}^j \subseteq
  \langle \Aut^+(T), G_{q,2}^i\rangle$. 

  There exists an element $\psi_*\in G_{q,2}^j$ with representative
  $\psi\colon T\setminus F_1\to T\setminus F_2$ inducing the same
  bijection of trees as $\phi$. Then $\psi^{-1}\phi\colon T\setminus
  F_1 \to T\setminus F_1$ preserves the trees of the forest
  $T\setminus F_1$, and therefore extends to an automorphism of $T$
  fixing pointwise the subtree $F_1$, and in particular the edge
  $e\subseteq F_1$. Thus, $\phi_*$ is a product of $\psi_*\in
  G^{j}_{q,2}$ and a type-preserving automorphism.
\end{proof}

Since the embedded copies of $G_{q,2}$ are conjugate by elements of
$\Aut^+(T)$, we may slightly abuse the notation and write $N_q =
\langle \Aut^+(T), G_{q,2}\rangle$.

\section{The simplicity of the Neretin groups}
\label{sec:simplicity}

In this section we present the proof of the simplicity of the Neretin
groups following \cite{Kapoudjian} using a method that was introduced
by Epstein in \cite{Epstein}. We remark that by using the group
topology one can also provide potentially simpler alternative proofs
which do not use Epstein's method. Nevertheless, the following lemmas
apply to a large variety of examples and are thus worth recalling.

We begin by two general lemmas about actions on topological
spaces. The setting we will consider will consist of a compact
Hausdorff topological space $X$, and a faithful group action by
homeomorphisms of a group $G$ on $X$.

\begin{lemma}
\label{restriction in normal subgroup}
Let $X$ and $G$ be as above, let $\mathcal{U}$ be a basis of $X$ on
which $G$ acts transitively, and let $1\ne H\triangleleft G$ be a
non-trivial normal subgroup of $G$. Then, for all $g\in G$ such that
$\supp(g)\subseteq V\in \mathcal{U}$ there exists an element $\rho\in
H$ such that $\rho |_V = g |_V$.
\end{lemma}

\begin{proof}
  Let $1\ne\alpha\in H$ be any non-trivial element of $H$. Let $x\in
  X$ be a point for which $\alpha^{-1}(x)\ne x$.  One can find a basis
  set $V_0\in\mathcal U$ such that $\alpha^{-1}(V_0)\cap
  V_0=\emptyset$.

  Assume first that $V=V_0$, and consider $\rho=[g,\alpha]=g\alpha
  g^{-1} \alpha^{-1}$. Since $\alpha\in H$ and $H$ is normal in $G$ we
  get that $\rho\in H$.  Moreover, by our assumption $\supp(g)
  \subseteq V$, thus when restricted to $V$ we see that $\rho |_V = g
  |_V$, as required.

  More generally, if $V\ne V_0$, we may find $h\in G$ such that
  $hV=V_0$. Now the element $g'=hgh^{-1}$ satisfies
  $\supp(hgh^{-1})\subseteq V_0$ and thus by the above we can find
  $\rho'\in H$ that satisfies $\rho' |_{V_0} = g' |_{V_0}$.  Since $H$
  is normal $\rho=h^{-1}\rho h\in H$, and $\rho$ satisfies $\rho |_{V}
  = g |_{V}$
\end{proof}

\begin{lemma} 
\label{commutator in normal subgroup}
Let $X$ and $G$ be as above, let $\mathcal{U}$ be a basis of $X$ on
which $G$ acts transitively, and let $1\ne H\triangleleft G$ be a
non-trivial normal subgroup of $G$ in which there exists
$\alpha_1,\alpha_2\in H$ such that for some $x\in X$ the points
$x,\alpha_1(x),\alpha_2(x)$ are distinct. Then, for all $g_1,g_2\in G$
such that $\supp(g_1),\supp(g_2)\subseteq V\in \mathcal{U}$ there
exist element $\rho_1,\rho_2\in H$ such that
$[g_1,g_2]=[\rho_1,\rho_2]$.
\end{lemma}

\begin{proof}
  Let $x\in X$ and $\alpha_1,\alpha_2\in H$ be as assumed.  One can
  find a basis set $V_0\in\mathcal U$ such that
  $V_0,\alpha_1(V_0),\alpha_2(V_0)$ are pairwise disjoint.

  As in the proof of the previous lemma we may assume up to
  conjugation that $V=V_0$, and consider $\rho_1=[g,\alpha_1]$ and
  $\rho_2=[g,\alpha_2]$. Again we have $\rho_1,\rho_2\in H$ as
  required.  Moreover, by our assumption $\supp(g_1),\supp(g_2)
  \subseteq V$, thus when restricted to $V$ we see that
  $[\rho_1,\rho_2] |_V = [g_1,g_2] |_V$. Moreover, $\rho_i$ ($i=1,2$)
  preserves the pairwise disjoint sets
  $V_0,\alpha_1(V_0),\alpha_2(V_0)$ and is supported on $V_0 \cup
  \alpha_i(V_0)$. It follows that $[\rho_1,\rho_2] |_{X\setminus V} =
  \rm{id} = [g_1,g_2] |_{X\setminus V}$. Thus overall we get
  $[g_1,g_2]=[\rho_1,\rho_2]$.
\end{proof}

\begin{remark}
\label{remark about lemma about commutators}
Note that in order to find such $\alpha_1,\alpha_2$ as in Lemma
\ref{commutator in normal subgroup}, by Lemma \ref{restriction in
  normal subgroup} it is enough to find two such elements in $G$ that
are supported on a basis set.
\end{remark}

We now apply the previous lemmas to prove the simplicity of the
Neretin group.

\begin{theorem}
  The Neretin group $N_q$ is simple.
\end{theorem}

\begin{proof}
  Let $1\ne H\triangleleft N_q$ be a nontrivial subgroup of the
  Neretin group. From Lemma \ref{lem:generating-set} the Neretin group
  is generated by $\Aut^+(T_q)$ and $G_{q,2}$. In fact, it is enough
  to take the commutator subgroup $G_{q,2}'$ of $G_{q,2}$, since
  $\Aut^+(T_q)\cap(G_{q,2}\setminus G_{q,2}')\ne\emptyset$
  whenever $G_{q,2}\ne G_{q,2}'$. From Theorems \ref{Tits
    simplicity} and \ref{thm:commutator-higman-thompson} the subgroups
  $\Aut^+(T_q)$ and $G_{q,2}'$ are simple. Thus, in order to prove
  the claim it is enough to show that $H\cap \Aut^+(T_q)\ne 1$ and
  $H\cap G_{q,2}'\ne 1$.

  We observe that the Neretin group acts faithfully by homeomorphisms
  on the boundary of the tree $T_q$, and acts transitively on the
  basis of ends of half trees. We complete the proof by Lemma
  \ref{commutator in normal subgroup} and \ref{remark about lemma
    about commutators} after finding two pairs of non-commuting
  elements in $\Aut^+(T_q)$ and $G_{q,2}$ that are supported in a
  half-tree.
\end{proof}

\bibliographystyle{plain}
\bibliography{neretin}

\begin{thebibliography}{1}

\bibitem{Brin2004}
Matthew Brin.
\newblock {Higher Dimensional Thompson Groups}.
\newblock {\em Geometriae Dedicata}, 108(1):163--192, 2004.

\bibitem{Epstein}
D.~B.~A. Epstein.
\newblock The simplicity of certain groups of homeomorphisms.
\newblock {\em Compositio Mathematica}, 22(2):165--173, 1970.

\bibitem{Hughes2009}
Bruce Hughes.
\newblock {Local similarities and the Haagerup property (with an appendix by
  Daniel S. Farley)}.
\newblock {\em Groups, Geometry, and Dynamics}, 3(2):299--315, 2009.

\bibitem{Kapoudjian}
Christophe Kapoudjian.
\newblock {Simplicity of Neretin's group of spheromorphisms}.
\newblock {\em Annales de l'institut Fourier}, 49(4):1225--1240, 1999.

\bibitem{Neretin1993}
Yu.~A. Neretin.
\newblock {On Combinatorial Analogs of the Group of Diffeomorphisms of the
  Circle}.
\newblock {\em Russian Academy of Sciences. Izvestiya Mathematics (in
  Russian)}, 41(2):337--349, April 1993.

\bibitem{Tits}
Jacques Tits.
\newblock Sur le groupe des automorphismes d’un arbre.
\newblock In {\em Essays on Topology and Related Topics}, pages 188--211.
  Springer Berlin Heidelberg, 1970.

\end{thebibliography}
\end{document}